\documentclass{amsart}
\usepackage{xcolor}
\usepackage{hyperref}
\hypersetup{
	colorlinks=true,
	citecolor=blue,
	linkcolor=blue,
	filecolor=magenta,      
	urlcolor=cyan,
}
\usepackage[capitalise,noabbrev]{cleveref}

\usepackage{amsmath,amscd}
\usepackage{amssymb}
\usepackage{mathtools}
\usepackage{stmaryrd}
\usepackage{url}
\usepackage{tikz-cd}
\usepackage{enumitem}

\usepackage{fullpage}

\usepackage{setspace}
\usepackage{microtype}

\usepackage{times}

\usepackage{tikz}

\usepackage{enumitem,kantlipsum}

\address{Department of Algebra, Faculty of Mathematics and Physics, Charles University in Prague, Sokolovsk\'a 83, 186 75 Praha, Czech Republic}

\email{shaul@karlin.mff.cuni.cz}

\newtheorem{thm}[equation]{Theorem}
\newtheorem*{thm*}{Theorem}
\newtheorem*{cor*}{Corollary}
\newtheorem*{dfn*}{Definition}

\newtheorem{cthm}{Theorem}

\newtheorem{que}[cthm]{Question}

\newtheorem{cor}[equation]{Corollary}
\newtheorem{prop}[equation]{Proposition}
\newtheorem{lem}[equation]{Lemma}

\theoremstyle{definition}
\newtheorem{dfn}[equation]{Definition}
\newtheorem{rem}[equation]{Remark}
\newtheorem{exa}[equation]{Example}

\newcommand{\opn}{\operatorname}
\newcommand{\cat}[1]{\operatorname{\mathsf{#1}}}

\newcommand{\mrm}[1]{\mathrm{#1}}
\newcommand{\mbb}[1]{\mathbb{#1}}

\newcommand{\K}{\mbb{K} \hspace{0.05em}}

\newcommand{\injdim}{\operatorname{inj\,dim}}
\newcommand{\projdim}{\operatorname{proj\,dim}}
\newcommand{\flatdim}{\operatorname{flat\,dim}}

\newcommand{\amp}{\operatorname{amp}}

\def\skewtimes{\ltimes\!}

\newcommand{\op}{\opn{op}}

\title{The finitistic dimension conjecture via DG-rings}
\author{Liran Shaul}
\dedicatory{Dedicated to Henning Krause on the occasion of his 60th birthday} 

\begin{document}
\begin{abstract}
Given an associative ring $A$,
we present a new approach for establishing the finiteness of the big finitistic projective dimension $\opn{FPD}(A)$.
The idea is to find a sufficiently nice non-positively graded differential graded ring $B$ such that $\mrm{H}^0(B) = A$ and such that $\opn{FPD}(B) < \infty$.
We show that one can always find such a $B$ provided that $A$ is noetherian and has a noncommutative dualizing complex.
We then use the intimate relation between $\cat{D}(B)$ and $\cat{D}(\mrm{H}^0(B))$ to deduce results about $\opn{FPD}(A)$.
As an application,
we generalize a recent sufficient condition of Rickard, for $\opn{FPD}(A) < \infty$ in terms of generation of $\cat{D}(A)$ from finite dimensional algebras over a field to all noetherian rings which admit a dualizing complex.
\end{abstract}

\numberwithin{equation}{section}
\maketitle

\setcounter{section}{-1}

\section{Introduction}

All rings in this paper are associative and unital, and modules by default are left modules. The finitistic projective dimension of a ring $A$, denoted by $\opn{FPD}(A)$ is the supremum of the projective dimension among all $A$-modules of finite projective dimension.
It is a major open problem whether this quantity is finite for artin algebras. See the survey \cite{Zim} for details.
One key reason for the importance of this conjecture is that it implies many other important homological conjectures.
The aim of this paper is to present a new approach for establishing the finiteness of the finitistic projective dimension of a ring $A$.

In a recent paper \cite{Rickard}, Rickard, resolving a conjecture of Keller,
showed that for a finite dimensional algebra $A$ over a field,
if the injective $A$-modules generate the unbounded derived category $\cat{D}(A)$, then $\opn{FPD}(A) < \infty$. 
This suggests that by enlarging the category that one works with,
here, from modules to the unbounded derived category,
one is able to obtain strong results about the finitistic dimension.
This is the approach taken in this paper,
and we use this idea in two different directions.

First, we enlarge the category of rings to the category of non-positively graded (in cohomological grading) differential graded rings.
Then, given a ring $A$, we construct a nicely behaved differential graded ring $B$ with the property that $\mrm{H}^0(B) = A$. Such an equality implies there are strong relations between the derived categories $\cat{D}(A)$ and $\cat{D}(B)$, because $\cat{D}(B)$ is equipped with a natural t-structure whose heart is equal to $\opn{Mod}(A)$.

As shown recently in \cite{DGFinite}, it is possible to extend the definition of the finitistic dimension to DG-rings. 
For a very large class of rings $A$, 
which includes all finite dimensional algebras over a field,
we show in \cref{cor:existDGFFD}  that one can find a DG-ring $B$ such that $\mrm{H}^0(B) = A$,
and such that moreover $\opn{FPD}(B) < \infty$.

This raises the following important question:
\begin{que}\label{theQuestion}
Let $B$ be a left noetherian DG-ring with bounded cohomology.
If $\opn{FPD}(B) < \infty$, does this imply that $\opn{FPD}(\mrm{H}^0(B))<\infty$?
\end{que}

A positive answer to this question will establish the big finitistic dimension conjecture for finite dimensional algebras over a field.
As we show in \cref{thm:commutative},
the results of \cite{DGFinite} imply
that \cref{theQuestion} has a positive answer provided that $B$ is assumed to be commutative. In fact, in that case, $\opn{FPD}(B) < \infty$ if and only if $\opn{FPD}(\mrm{H}^0(B))<\infty$.

To study \cref{theQuestion} further, we limit the scope of our investigation to DG-rings $B$ for which one can lift any $\mrm{H}^0(B)$-module $\overline{M}$ of finite projective dimension to a DG-module $M$ over $B$ such that $\mrm{H}^0(B) \otimes^{\mrm{L}}_B M = \overline{M}$. 
The DG-ring we construct in \cref{cor:existDGFFD} satisfies this condition.
We show in \cref{thm:boundlift} that for such a DG-ring $B$,
for which a lifting exists and $\opn{FPD}(B) < \infty$,
the question whether $\opn{FPD}(\mrm{H}^0(B)) < \infty$ is actually equivalent to certain properties of the lifting.

To construct such a DG-ring $B$,
for a given ring $A$,
we make use of the work of Yekutieli \cite{Yek} about noncommutative dualizing complexes.
Following work of J{\o}rgensen \cite{Jor} in commutative algebra, we study for a noncommutative ring $A$ with a dualizing complex $R$,
the trivial extension DG-ring $A \skewtimes R$.
We then show in \cref{thm:trivialExt} that it is always Gorenstein.
Generalizing a result of Bass \cite{Bass} from rings to DG-rings, we prove in \cref{thm:injDimFFD} that any such Gorenstein DG-ring has finite finitistic dimension.

We then use the differential graded techniques to generalize Rickard's theorem. 
For our argument to work, the unbounded derived category is too small,
and instead we have to consider the three homotopy categories $\cat{K}(\opn{Inj}(A))$, $\cat{K}(\opn{Proj}(A))$ and $\cat{K}(\opn{Flat}(A))$ associated to a ring $A$.
Working with these larger triangulated categories,
and using the Iyengar-Krause covariant Grothendieck duality theorem \cite{InKr},
as well as some deep results of Neeman \cite{Neeman2008} about $\cat{K}(\opn{Flat}(A))$
we obtain a very broad generalization of the mentioned theorem of Rickard \cite[Theorem 4.3]{Rickard}. Precisely, we show in \cref{cor:main} the following:

\begin{cthm}\label{cthm:main}
Let $A$ be a ring which is left and right noetherian and has a dualizing complex.
Assume that the localizing subcategory generated by the injective $A$-modules is equal to $\cat{D}(A)$.
Then $\opn{FPD}(A) < \infty$.
\end{cthm}

Finally, we show in \cref{exa:nessDual} that the condition that $A$ has a dualizing complex is necessary. In the absence of a dualizing complex,
the fact that the injectives generate does not imply that the ring has finite finitistic dimension. It follows that in general \cref{cthm:main} cannot be improved, and that it precisely identifies the class of rings for which generation by injectives implies finite finitistic dimension.

\section{Finitistic dimensions over differential graded rings}

A non-positive differential graded ring is a graded ring of the form 
\[
A = \bigoplus_{n=-\infty}^0 A^n
\]
equipped with a differential of degree $+1$ which satisfies a Leibniz rule. All DG-rings in this paper are assumed to be non-positive. We refer the reader to \cite{Kel,YeBook} for more details about DG-rings and their derived categories.
DG-modules over DG-rings will be assumed to be left DG-modules by default.
The derived category of left DG-modules over a DG-ring $A$ will be denoted by $\cat{D}(A)$. It is a triangulated  category. Its full subcategory consisting of DG-modules with bounded cohomology will be denoted by $\cat{D}^{\mrm{b}}(A)$.
We say that $A$ is left noetherian if the ring $\mrm{H}^0(A)$ is left noetherian and for all $i<0$ the left $\mrm{H}^0(A)$-module $\mrm{H}^i(A)$ is finitely generated.
We further say that $A$ has bounded cohomology if $\mrm{H}^i(A) = 0$ for all $i\ll 0$.
Given $M \in \cat{D}(A)$,
we define its infimum, supremum and amplitude by
\[
\inf(M) = \inf\{n\mid \mrm{H}^n(M) \ne 0\},
\quad
\sup(M) = \sup\{n\mid \mrm{H}^n(M) \ne 0\},
\]
and $\amp(M) = \sup(M) - \inf(M)$.

For a pair of left DG-modules $M,N \in \cat{D}(A)$,
we set $\opn{Ext}^i_A(M,N) := \mrm{H}^i\left(\mrm{R}\opn{Hom}_A(M,N)\right)$.
If $M$ is a right DG-module over $A$ and $N$ is a left DG-module over $A$, we further set $\opn{Tor}^A_n(M,N) := \mrm{H}^{-n}(M\otimes^{\mrm{L}}_A N)$.

Using these definitions,
the projective, injective and flat dimension of complexes over rings carry over to the differential graded case. 
For $M \in \cat{D}^{\mrm{b}}(A)$ we define them by the formulas
\[
\projdim_A(M) = \inf\{n \in \mathbb{Z} \mid \opn{Ext}_A^i(M,N) = 0 \text{ for any $N \in \cat{D}^\mrm{b}(A)$ and any $i > n + \sup(N)$}\},
\]
\[
\injdim_A(M) = \inf\{n \in \mathbb{Z} \mid \opn{Ext}_A^i(N,M) = 0 \text{ for any $N \in \cat{D}^\mrm{b}(A)$ and any $i > n - \inf(N)$}
\},
\]
and
\[
\flatdim_A(M) = \inf\{n \in \mathbb{Z} \mid \opn{Tor}^A_i(N,M) = 0\text{ for any $N \in \cat{D}^\mrm{b}(A^{\op})$ and any $i > n - \inf(N)$}\}.
\]
As shown in \cite[Proposition 1.4]{DGFinite},
it is enough to test these on DG-modules with amplitude $0$.

Next is a key definition for this paper, following \cite{DGFinite}.

\begin{dfn}
Let $A$ be a non-positive DG-ring with bounded cohomology.
We define the finitistic projective dimension of $A$,
denoted by $\opn{FPD}(A)$, by the formula
\[
\opn{FPD}(A) = \sup\{\projdim_A(M) + \inf(M) \mid M \in \cat{D}^{\mrm{b}}(A), \projdim_A(M) < \infty\}.
\]
\end{dfn}

For an ordinary ring, one can show that this definition coincides with the usual definition of the finitistic dimension in terms of projective dimension of modules. The term $\inf(M)$ in it is used to normalize the situation for complexes which are not necessarily modules: the quantity $\projdim_A(M) + \inf(M)$ is invariant under shifts.

The following result, which is the raison d'\^etre of this paper, is about \textbf{commutative} DG-rings,
and follows easily from the results of \cite{DGFinite}:

\begin{thm}\label{thm:commutative}
Let $A$ be a commutative noetherian DG-ring with bounded cohomology. 
Then $\opn{FPD}(A) < \infty$ if and only if $\opn{FPD}(\mrm{H}^0(A)) < \infty$.
\end{thm}
\begin{proof}
Since $\mrm{H}^0(A)$ is a commutative noetherian ring,
by \cite{RG}, it holds that $\opn{FPD}(\mrm{H}^0(A)) < \infty$ if and only if $\dim(\mrm{H}^0(A)) < \infty$.
If $\dim(\mrm{H}^0(A)) < \infty$ then by \cite[Theorem 6.2]{DGFinite} it holds that $\opn{FPD}(A) < \infty$,
while if $\dim(\mrm{H}^0(A)) = \infty$,
then it follows from \cite[Theorem 5.6]{DGFinite} that $\opn{FPD}(A) = \infty$.
\end{proof}

\section{Formal reduction results for the finitistic dimension}

Let $A$ be a non-positive DG-ring.
In this section we discuss relations between the finitistic projective dimension of $A$ and the finitistic projective dimension of $\mrm{H}^0(A)$,
in an attempt to imitate \cref{thm:commutative} in noncommutative contexts.
To obtain such a relation, 
we define the following conditions on $A$,
which allows one to lift certain $\mrm{H}^0(A)$-modules to DG-modules over $A$.

\begin{dfn}
Let $A$ be a non-positive DG-ring.
\begin{enumerate}
\item We say that $A$ \textit{lifts modules of finite projective dimension} if for any $\overline{M} \in \opn{Mod}(\mrm{H}^0(A))$ such that 
\[
\projdim_{\mrm{H}^0(A)}(\overline{M}) < \infty,
\]
there exists $M \in \cat{D}^{\mrm{b}}(A)$ such that $\mrm{H}^0(A)\otimes^{\mrm{L}}_A M \cong \overline{M}$.
In that case we call $M$ a lifting of $\overline{M}$.
\item In this situation, we say that $A$ has \textit{bounded lifting} if there exists a bound $N \in \mathbb{N}$ such that for any such $\overline{M}$ there exists a lifting $M$ with $\amp(M)\le N$.
\end{enumerate}
\end{dfn}

Here are a couple of basic facts about these notions.
\begin{prop}\label{eqn:projdim}
Let $A$ be a non-positive DG-ring.
\begin{enumerate}
\item If $M \in \cat{D}^{\mrm{b}}(A)$ is a lifting of a module $\overline{M} \in \opn{Mod}(\mrm{H}^0(A))$ of finite projective dimension,
then $\projdim_A(M) = \projdim_{\mrm{H}^0(A)}(\overline{M})$.
In particular, $M$ has finite projective dimension over $A$.
\item If $A$ lifts modules of finite projective dimension,
then the lifting is bounded if and only if there exists some $N \in \mathbb{N}$ such that for any such $\overline{M}$ there exists a lifting $M$ with $\inf(M)\ge -N$.
\end{enumerate}
\end{prop}
\begin{proof}
\begin{enumerate}
\item By \cite[Corollary 1.5]{DGFinite}, 
there is an equality 
\[
\projdim_A(M) = \projdim_{\mrm{H}^0(A)}(\mrm{H}^0(A)\otimes^{\mrm{L}}_A M) = \projdim_{\mrm{H}^0(A)}(\overline{M}) < \infty.
\]
\item This follows from the equality
\[
0 = \sup(\overline{M}) = \sup\left(\mrm{H}^0(A)\otimes^{\mrm{L}}_A M\right) = \sup(M)
\]
which implies that $\amp(M) = -\inf(M)$.
\end{enumerate}
\end{proof}

In general, it seems difficult to prove the existence of lifting for all modules of finite projective dimension.
There is however one important special case where lifting is guaranteed.

\begin{dfn}
Let $A$ be a non-positive DG-ring.
We say that the natural projection map $\pi_A:A \to \mrm{H}^0(A)$ has a section if there is a map of DG-rings $\tau_A:\mrm{H}^0(A) \to A$ such that 
$\pi_A \circ \tau_A = 1_{\mrm{H}^0(A)}$.
\end{dfn}

In a later section we will study trivial extension DG-rings, and for them $\pi_A$ will always have a section.
The importance of this notion for us is that the existence of a section guarantees lifting: 

\begin{prop}
Let $A$ be a non-positive DG-ring with bounded cohomology, 
and suppose that the natural projection map $\pi_A:A \to \mrm{H}^0(A)$ has a section $\tau_A$.
Then $A$ lifts modules of finite projective dimension.
\end{prop}
\begin{proof}
Let $\overline{M} \in \opn{Mod}(\mrm{H}^0(A))$ be an $\mrm{H}^0(A)$-module of finite projective dimension.
Set $M = A\otimes^{\mrm{L}}_{\mrm{H}^0(A)} \overline{M}$.
Since $\overline{M}$ has finite projective dimension over $\mrm{H}^0(A)$,
in particular $\flatdim_{\mrm{H}^0(A)}(\overline{M}) < \infty$, so the fact that $\amp(A) < \infty$ implies that $M \in \cat{D}^{\mrm{b}}(A)$.
The fact that $\pi_A \circ \tau_A = 1_{\mrm{H}^0(A)}$ then implies that
\[
\mrm{H}^0(A)\otimes^{\mrm{L}}_A M = 
\mrm{H}^0(A)\otimes^{\mrm{L}}_A A\otimes^{\mrm{L}}_{\mrm{H}^0(A)} \overline{M}
\cong \overline{M},
\]
which shows that $A$ lifts modules of finite projective dimension.
\end{proof}

We now show that the above conditions allow us to obtain relations between the finiteness of $\opn{FPD}(A)$ and $\opn{FPD}(\mrm{H}^0(A))$.

\begin{thm}\label{thm:boundlift}
Let $A$ be a non-positive DG-ring with bounded cohomology.
Suppose that $A$ lifts modules of finite projective dimension and $\opn{FPD}(A) < \infty$.
\begin{enumerate}
\item 
If $A$ has bounded lifting then $\opn{FPD}(\mrm{H}^0(A)) < \infty$.
\item If $\pi_A:A \to \mrm{H}^0(A)$ has a section,
then the converse holds: $A$ has bounded lifting if and only if $\opn{FPD}(\mrm{H}^0(A)) < \infty$.
\end{enumerate}
\end{thm}
\begin{proof}
Suppose that $A$ has bounded lifting with a bound $N \in \mathbb{N}$. Let $n = \opn{FPD}(A)$,
and let $\overline{M} \in \opn{Mod}(\mrm{H}^0(A))$ be an $\mrm{H}^0(A)$-module with $\projdim_{\mrm{H}^0(A)}(\overline{M}) < \infty$.
By assumption, there exist a lifting  $M \in \cat{D}^{\mrm{b}}(A)$ such that $\mrm{H}^0(A)\otimes^{\mrm{L}}_A M \cong \overline{M}$ 
and $\inf(M) \ge -N$.
Moreover, by \cref{eqn:projdim},
we know that $\projdim_A(M) < \infty$.
Hence, by the definition of finitistic dimension over $A$, 
it follows that
\[
\projdim_A(M) \le \opn{FPD}(A) - \inf(M) \le n+N.
\]
Since $\projdim_A(M) = \projdim_{\mrm{H}^0(A)}(\overline{M})$,
we deduce that $\opn{FPD}(\mrm{H}^0(A)) \le n+N < \infty$.

Conversely, suppose that $\pi_A:A \to \mrm{H}^0(A)$ has a section. If $N = \opn{FPD}(\mrm{H}^0(A)) < \infty$,
then any $\mrm{H}^0(A)$-module $\overline{M}$ of finite projective dimension satisfies 
\[
\flatdim_{\mrm{H}^0(A)}(\overline{M}) \le \projdim_{\mrm{H}^0(A)}(\overline{M}) \le N,
\]
which implies that the lifting
$M = A\otimes^{\mrm{L}}_{\mrm{H}^0(A)} \overline{M}$
satisfies $\inf(M) \ge -N+\inf(A)$.
Since $A$ has bounded cohomology, $\inf(A) > -\infty$,
so the latter quantity is finite, showing that $A$ has bounded lifting.
\end{proof}

\begin{rem}
It is possible that the notion of a twisted complex,
as studied in \cite{GLV} can be used to obtain better control on the lifting of an $\mrm{H}^0(A)$-module of finite projective dimension to $\cat{D}^{\mrm{b}}(A)$.
\end{rem}

\section{Finitistic dimension of DG-rings with finite injective dimension}

In \cite[Proposition 4.3]{Bass},
Bass showed that if $A$ is a left noetherian ring with $\injdim_A(A) < \infty$ then 
\[
\opn{FPD}(A) \le \injdim_A(A) < \infty.
\]
The aim of this section is to generalize this result of Bass to the DG-setting. Such Gorenstein conditions on DG-rings were studied in detail in \cite{FJ}.

We need several preliminary results concerning injective dimension over DG-rings. Key references for this topic are \cite{Min,ShINJ}.
One important notion we will need from loc. cit, 
is the notion of a derived injective DG-module.
These, by definition, are left DG-modules $I$ such that either $I \cong 0$,
or $\inf(I) = \injdim_A(I) = 0$.
We denote by $\opn{Inj}(A)$ the full subcategory of $\cat{D}(A)$ consisting of derived injective DG-modules.

\begin{lem}\label{lem:injdimineqal}
Let $A$ be a DG-ring,
and let $M \in \cat{D}^{+}(A)$ be a non-zero DG-module.
Then $\injdim_A(M) \ge \inf(M)$.
\end{lem}
\begin{proof}
This follows from the definition of injective dimension and the fact that 
\[
\opn{Ext}^{\inf(M)}_A(\mrm{H}^0(A),M) = \mrm{H}^{\inf(M)}(M) \ne 0.
\]
\end{proof}

We now show that over left noetherian DG-rings,
direct sums do not increase injective dimension.

\begin{thm}\label{thm:dirSumInjDim}
Let $A$ be a left noetherian DG-ring. 
Let $\{M_{\alpha}\}_{\alpha \in I}$ be a collection of  DG-modules over $A$, 
and suppose that for all $\alpha \in I$ it holds that $\injdim_A(M_{\alpha}) \le n$.
Let $M = \oplus_{\alpha \in I} M_{\alpha}$,
and assume that $\inf(M) > -\infty$.
Then $\injdim_A(M) \le n$.
\end{thm}
\begin{proof}
By shifting if necessary, 
we may assume without loss of generality that
\[
\inf(M) = \inf\{\inf(M_{\alpha})\mid \alpha \in I\} = 0.
\]
In particular, this implies by \cref{lem:injdimineqal} that $n \ge 0$.
We prove the result by induction on $n$.
If $n = 0$,
this implies by \cref{lem:injdimineqal} that for all $\alpha \in I$,
$\inf(M_{\alpha}) = \injdim_A(M_{\alpha}) = 0$,
so that $M_{\alpha} \in \opn{Inj}(A)$.
Thus, in this case, the statement is that $\opn{Inj}(A)$ is closed under direct sums, which is the derived Bass-Papp theorem \cite[Theorem 6.6]{ShINJ} (or \cite[Theorem 3.30]{Min}).
Suppose now that $n>0$.
For each $\alpha \in I$,
let us choose an element $J_{\alpha} \in \opn{Inj}(A)$ and a map $f_{\alpha}:M_{\alpha} \to J_{\alpha}$ in $\cat{D}(A)$ such that $\mrm{H}^0(f_{\alpha})$ is injective.
Such a map exists by \cite[Corollary 3.16]{Min}.
Note that if $\inf(M_{\alpha})>0$, 
we can simply take $J_{\alpha}=0 \in \opn{Inj}(A)$.
The map $f_{\alpha}$ induces a distinguished triangle
\[
M_{\alpha} \xrightarrow{f_{\alpha}} J_{\alpha} \to K_{\alpha} \to M_{\alpha}[1]
\]
in $\cat{D}(A)$. 
The exact sequence of $\mrm{H}^0(A)$-modules
\[
0 = \mrm{H}^{-1}(J_{\alpha}) \to \mrm{H}^{-1}(K_{\alpha}) \to \mrm{H}^0(M_{\alpha}) \xrightarrow{\mrm{H}^0(f_{\alpha})} \mrm{H}^0(N_{\alpha})
\]
and the fact that $\mrm{H}^0(f_{\alpha})$ is injective implies that $\inf(K_{\alpha}) \ge 0$.
Given any $T \in \cat{D}(A)$ with $\inf(T) = \sup(T) = 0$,
applying the functor $\mrm{R}\opn{Hom}_A(T,-)$ to the above distinguished triangle gives for any $i\in \mathbb{Z}$ the following exact sequence of abelian groups
\[
\opn{Ext}^i_A(T,J_{\alpha}) \to \opn{Ext}^i_A(T,K_{\alpha})
\to \opn{Ext}^{i+1}_A(T,M_{\alpha}) \to \opn{Ext}^{i+1}_A(T,J_{\alpha})
\]
If $i>0$, then the fact that $J_{\alpha} \in \opn{Inj}(A)$ implies by \cite[Theorem 4.10]{ShINJ} that
\[
\opn{Ext}^i_A(T,J_{\alpha}) = \opn{Ext}^{i+1}_A(T,J_{\alpha}) = 0.
\]
Hence, for all $i>0$ there is an isomorphism
\[
\opn{Ext}^i_A(T,K_{\alpha}) \cong \opn{Ext}^{i+1}_A(T,M_{\alpha}),
\]
which shows that $\injdim_A(K_{\alpha}) = \injdim_A(M_{\alpha}) - 1$.
We now wish to use the induction hypothesis by replacing $M_{\alpha}$ with $K_{\alpha}$.
Indeed, the above shows that for all $\alpha \in I$ it holds that $\injdim_A(K_{\alpha}) \le n-1$.
It may happen now that 
\[
\inf\{\inf(K_{\alpha})\mid \alpha \in I\} > 0,
\]
so as above, we shift $K_{\alpha}$ if needed so this number becomes zero. Shifting like this can only decrease injective dimension, so the induction hypothesis is indeed satisfied,
and it follows by induction that 
\[
\injdim_A(\oplus_{\alpha \in I} K_{\alpha}) \le n-1.
\]

According to \cite[tag 0CRG]{SP},
it holds that
\[
\oplus_{\alpha \in I} M_{\alpha} \xrightarrow{\oplus_{\alpha \in I} f_{\alpha}} \oplus_{\alpha \in I} J_{\alpha} \to \oplus_{\alpha \in I} K_{\alpha} \to \oplus_{\alpha \in I} M_{\alpha}[1]
\]
is a distinguished triangle in $\cat{D}(A)$. 
Applying the derived Bass-Papp theorem \cite[Theorem 6.6]{ShINJ} again, we have that 
$\oplus_{\alpha \in I} J_{\alpha} \in \opn{Inj}(A)$.
Arguing again as above, 
this distinguished triangle implies that
\[
\injdim_A(\oplus K_{\alpha}) = \injdim_A(\oplus M_{\alpha}) - 1,
\]
which shows that $\injdim_A(M) \le n$ as claimed.
\end{proof}

\begin{cor}\label{cor:homDirSum}
In the situation of \cref{thm:dirSumInjDim},
for any $N \in \cat{D}_{\mrm{f}}(A)$ the natural map
\[
\bigoplus_{\alpha \in I}\left(\mrm{R}\opn{Hom}_A(N,M_{\alpha})\right) \to
\mrm{R}\opn{Hom}_A(N,\bigoplus_{\alpha \in I} M_{\alpha})
\]
is an isomorphism.
\end{cor}
\begin{proof}
Considering the two functors
\[
F(-):= \bigoplus_{\alpha \in I}\left(\mrm{R}\opn{Hom}_A(-,M_{\alpha})\right)
\]
and
\[
G(-):= \mrm{R}\opn{Hom}_A(-,\bigoplus_{\alpha \in I} M_{\alpha})
\]
there is a natural map $\eta:F \to G$, 
and it is clear that $\eta_A:F(A) \to G(A)$ is an isomorphism.
By \cref{thm:dirSumInjDim},
the functors $F$ and $G$ both have finite cohomological dimension.
Hence, the result follows from the lemma on way-out functors \cite[Theorem 2.11]{YeDual}.
\end{proof}

\begin{cor}\label{cor:injdimFree}
Let $A$ be a left noetherian DG-ring with $\injdim_A(A) < \infty$.
Then for any set $I$, 
it holds that
\[
\injdim_A\left(\bigoplus_{\alpha \in I} A\right) = \injdim_A(A).
\]
\end{cor}
\begin{proof}
Set $n = \injdim_A(A)$.
Let $F = \bigoplus_{\alpha \in I} A$.
By \cref{thm:dirSumInjDim}, we know that $\injdim_A(F)\le \injdim_A(A)$.
By \cite[Lemma 2.4]{DGFinite}, 
there exist a left ideal $\overline{J} \subseteq \mrm{H}^0(A)$
such that $\opn{Ext}^n_A(\mrm{H}^0(A)/\overline{J},A) \ne 0$.
Hence, \cref{cor:homDirSum} implies that 
$\opn{Ext}^n_A(\mrm{H}^0(A)/\overline{J},F) \ne 0$,
which shows that $\injdim_A(F) = n$.
\end{proof}

Here is the main result of this section.

\begin{thm}\label{thm:injDimFFD}
Let $A$ be a left noetherian DG-ring with bounded cohomology.
Then $\opn{FPD}(A) \le \injdim_A(A)$.
In particular, if $\injdim_A(A) < \infty$
then $\opn{FPD}(A) < \infty$.
\end{thm}
\begin{proof}
There is nothing to prove if $\injdim_A(A) = \infty$,
so assume it is finite.
Let $M \in \cat{D}^{\mrm{b}}(A)$ be a non-zero DG-module such that $\projdim_A(M) < \infty$.
By shifting, we may assume that $\inf(M) = 0$,
and it is thus enough to show that $\injdim_A(A) \ge \projdim_A(M)$.
Assume $n = \projdim_A(M)$,
and let $N \in \cat{D}(A)$ be such that $\inf(N) = \sup(N) = 0$, and moreover $\opn{Ext}^n_A(M,N) \ne 0$.
Such a DG-module exists by \cite[Proposition 1.4]{DGFinite}.
Using \cite[Lemma 2.8]{Min},
we may find an index set $I$ such that for 
$F = \oplus_{\alpha \in I} A$,
there is a map $\varphi:F \to N$ such that $\mrm{H}^0(\varphi)$ is surjective. 
We embed this map in a distinguished triangle  
\[
K \to F \xrightarrow{\varphi} N \to K[1]
\]
in $\cat{D}(A)$.
Then it follows from surjectivity of $\mrm{H}^0(\varphi)$ that $\sup(K) \le 0$, so our assumption on the projective dimension of $M$ implies that $\opn{Ext}^{n+1}_A(M,K) = 0$.
Applying the functor $\mrm{R}\opn{Hom}_A(M,-)$ to the above distinguished triangle gives the exact sequence
\[
\opn{Ext}^n_A(M,F) \to \opn{Ext}^n_A(M,N) \to \opn{Ext}^{n+1}_A(M,K) = 0.
\]
Hence, the fact that $\opn{Ext}^n_A(M,N) \ne 0$
implies that $\opn{Ext}^n_A(M,F) \ne 0$.
This, and the definition of injective dimension implies that
\[
\injdim_A(F) \ge n-\inf(M) = n.
\]
By \cref{cor:injdimFree} this implies that $\injdim_A(A) \ge n$, establishing the result.
\end{proof}

\section{Trivial extension DG-rings and dualizing complexes}

The aim of this section is to generalize one of the main results of \cite{Jor} to a noncommutative setting.
J{\o}rgensen's result we seek to generalize says that if $A$ is a commutative noetherian local ring,
and if $R$ is a dualizing complex over $A$ with $\sup(R) \le 0$, then the trivial extension DG-ring $A \skewtimes R$ is a Gorenstein DG-ring. Since one may consider $A$ as a quotient of $A \skewtimes R$, this result says that any commutative noetherian local ring with a dualizing complex is a quotient of a Gorenstein DG-ring. More generally, Kawasaki, solving a conjecture of Sharp, showed in \cite[Corollary 1.4]{Kaw} that any such ring is a quotient of an ordinary Gorenstein local ring.

We prove in this section a noncommutative version of this result. 
First, we define the notion of a trivial extension DG-ring in the noncommutative setting.

\begin{dfn}
Let $A$ be a ring, and let $M$ be a complex of bimodules over $A$ with the property that $\sup(M) \le 0$.
The trivial extension DG-ring $A \skewtimes M$ is defined as follows. 
As a graded abelian group, we let
$A \skewtimes M = A \oplus M$.
The multiplication is defined by
\[
\begin{bmatrix}
a_1 \\
m_1
\end{bmatrix}
\cdot 
\begin{bmatrix}
a_2 \\
m_2
\end{bmatrix} 
=
\begin{bmatrix}
a_1 \cdot a_2\\
a_1 \cdot m_2 + m_1\cdot a_2
\end{bmatrix}
\]
The differential of $A \skewtimes M$ is the differential of $M$. 
The differential on $A$ is $0$.
\end{dfn}

It follows from the definitions that $A \skewtimes M$ is a non-positive DG-ring. If we further assume that $\sup(M) < 0$, then it holds that $\mrm{H}^0(A \skewtimes M) = A$.
The trivial extension DG-ring comes equipped with a natural map of DG-rings $\tau_{A,M}:A \to A \skewtimes M$ given by
\[
\tau_{A,M}: A \to A \skewtimes M \quad a \mapsto \begin{bmatrix} a \\ 0 \end{bmatrix}
\]
When $\sup(M) < 0$, it follows that the natural map
\[
\pi_{A \skewtimes M}:A\skewtimes M\to \mrm{H}^0(A \skewtimes M) = A
\]
satisfies $\pi_{A \skewtimes M} \circ \tau_{A,M} = 1_A$.
In other words, $\pi_{A \skewtimes M}$ has a section.

Restriction along the map $\tau_{A,M}$ allows one to view $A \skewtimes M$ as a complex of $A$-modules. 
We observe that as such, it is equal to the complex $A \oplus M$.

Next, we wish to study trivial extensions of rings by dualizing complexes. To do this, we recall this important notion.

The notion of a dualizing complex was first introduced by Grothendieck (with the details spelled out by Hartshorne) in an algebraic geometry context \cite{RD} in order to prove a duality theorem which was valid for not necessarily Cohen-Macaulay schemes.
In noncommutative contexts, this notion was first introduced by Yekutieli in \cite{Yek}. This is the notion we need in this paper, so we now recall the definition.

Let $A$ be a ring.
Given an $A-A$-bimodule $M$,
or more generally,
a complex of $A-A$-bimodules $M$,
we denote by $\opn{Res}_A(M)$ the restriction functor which forgets the right $A$-structure. Thus, $\opn{Res}_A(M) \in \cat{D}(A)$. Similarly, forgetting the left $A$-structure,
we let $\opn{Res}_{A^{\op}}(M) \in \cat{D}(A^{\op})$.

We call a ring $A$ noetherian if it is both left noetherian and right noetherian.

\begin{dfn}
Let $A$ be a noetherian ring.
A dualizing complex $R$ over $A$ is a complex of $R$ of $A-A$-bimodules which satisfies the following properties:
\begin{enumerate}
\item The complexes $\opn{Res}_A(R)$ and $\opn{Res}_{A^{\op}}(R)$ are bounded complexes of injective $A$-modules and injective $A^{\op}$-modules respectively.
\item The complexes $\opn{Res}_A(R)$ and $\opn{Res}_{A^{\op}}(R)$ have finitely generated cohomology over $A$ and $A^{\op}$ respectively.
\item The natural maps
\[
A \to \mrm{R}\opn{Hom}_A(R,R), \quad A \to \mrm{R}\opn{Hom}_{A^{\op}}(R,R)
\]
are isomorphisms.
\end{enumerate}
\end{dfn}

This definition is a direct generalization of the commutative case.

\begin{rem}
The above definition is slightly different than the definition given in \cite{Yek}.
The main difference is that in loc. cit,
one only assumes that  $\opn{Res}_A(R)$ and $\opn{Res}_{A^{\op}}(R)$ have finite injective dimension.
However, in that context, $A$ is an algebra over a field,
so by \cite[Proposition 2.4]{Yek},
both definitions are equivalent.
\end{rem}

\begin{exa}
Let $\K$ be a field,
and let $A$ be a finite dimensional $\K$-algebra.
Then $DA = \opn{Hom}_{\K}(A,\K)$ is a dualizing complex over $A$.
\end{exa}

There are many other examples of rings which admit dualizing complexes. We refer the reader to \cite{vdb,Yek,YZ,YZ2} for various  existence results.

We now fulfill the goal of this section,
and generalize to the noncommutative setting one direction of \cite[Theorem 2.2]{Jor}. Our proof is based on the commutative case.

\begin{thm}\label{thm:trivialExt}
Let $A$ be a noetherian ring,
let $R$ be a dualizing complex over $A$,
and suppose that $\sup(R) < 0$.
Letting $B = A \skewtimes R$ be the trivial extension DG-ring,
it holds that $\injdim_B(B) < \infty$
and $\injdim_{B^{\op}}(B^{\op}) < \infty$. 
\end{thm}
\begin{proof}
By symmetry, it is enough to show that $\injdim_B(B) < \infty$.
Let 
\[
N := \mrm{R}\opn{Hom}_A(B,\opn{Res}_A(R)) = \opn{Hom}_A(B,\opn{Res}_A(R)) \in \cat{D}(B)
\]
where the equality follows from the fact that $\opn{Res}_A(R)$ is a bounded complex of injective $A$-modules.
By definition, the $B$-module structure of $N$ is given by
\[
(b\cdot n) (c) = (-1)^{|b|(|n|+|c|)}\cdot n(c\cdot b)
\]
for homogeneous elements $b \in B$, $n \in N$ and $c \in B$.
The adjunction isomorphism
\[
\mrm{R}\opn{Hom}_B(-,N) = \mrm{R}\opn{Hom}_B(-,\mrm{R}\opn{Hom}_A(B,\opn{Res}_A(R))) \cong
\mrm{R}\opn{Hom}_A(-,\opn{Res}_A(R))
\]
and the fact that $\injdim_A(\opn{Res}_A(R)) < \infty$ implies that $\injdim_B(N) < \infty$.
To complete the proof, we will show that there is an isomorphism $B \cong N$ in $\cat{D}(B)$.
There is an element $\epsilon \in N$ given by
\[
\epsilon 
\begin{bmatrix}
a \\
r
\end{bmatrix}
= r.
\]
Using this element, we define a morphism
\[
\Psi: B \to N, \quad  \Psi(b) = b\cdot \epsilon.
\]
It is clear that $\Psi$ is a morphism in $\cat{D}(B)$.
The proof will be complete once we show that $\Psi$ is an isomorphism. 
To do that, we may apply the forgetful functor $\cat{D}(B) \to \cat{D}(A)$, 
and it is enough to show that as a morphism in $\cat{D}(A)$,
the map $\Psi$ is an isomorphism.
Working over $A$, we may write $B = A \oplus \opn{Res}_A(R)$,
and hence,
\[
N = \opn{Hom}_A(B,\opn{Res}_A(R)) = \opn{Hom}_A(A,\opn{Res}_A(R)) \oplus \opn{Hom}_A(\opn{Res}_A(R),\opn{Res}_A(R)).
\]
Using this decomposition, we compute
\begin{equation}\label{eq:comp1}
\left(
\Psi 
\begin{bmatrix}
a\\
0
\end{bmatrix}
\right)
\left(
\begin{bmatrix}
a' \\
r'
\end{bmatrix}
\right) 
=
\left(
\begin{bmatrix}
a\\
0
\end{bmatrix}
\cdot \epsilon
\right)
\left(
\begin{bmatrix}
a' \\
r'
\end{bmatrix}
\right) =
\epsilon
\left(
\begin{bmatrix}
a' \\
r'
\end{bmatrix}
\cdot 
\begin{bmatrix}
a\\
0
\end{bmatrix}
\right)
=
\epsilon
\begin{bmatrix}
a'\cdot a\\
r'\cdot a
\end{bmatrix}
= r'\cdot a
\end{equation}
and
\begin{equation}\label{eq:comp2}
\left(
\Psi 
\begin{bmatrix}
0\\
r
\end{bmatrix}
\right)
\left(
\begin{bmatrix}
a' \\
r'
\end{bmatrix}
\right) 
=
\left(
\begin{bmatrix}
0\\
r
\end{bmatrix}
\cdot \epsilon
\right)
\left(
\begin{bmatrix}
a' \\
r'
\end{bmatrix}
\right) =
\epsilon
\left(
\begin{bmatrix}
a' \\
r'
\end{bmatrix}
\cdot 
\begin{bmatrix}
0\\
r
\end{bmatrix}
\right)
=
\epsilon
\begin{bmatrix}
0\\
a' \cdot r
\end{bmatrix} =
a' \cdot r.
\end{equation}
It is clear that the map
\[
\Phi_1:\opn{Res}_A(R) \to \opn{Hom}_A(A,\opn{Res}_A(R)),
\quad
\left(\Phi_1(r)\right)(a) = a\cdot r
\]
is an isomorphism.
The fact that $R$ is a dualizing complex over $A$ and that $\opn{Res}_A(R)$ is a bounded complex of injective $A$-modules, implies that the map
\[
\Phi_2:A \to \opn{Hom}_A(\opn{Res}_A(R),\opn{Res}_A(R)),
\quad
\left(\Phi_2(a)\right)(r) = r \cdot a
\]
is bijective in cohomology.
The computations \cref{eq:comp1} and \cref{eq:comp2} show that using the decomposition
\[
\Psi: \opn{Res}_A(R) \oplus A \to
\opn{Hom}_A(A,\opn{Res}_A(R)) \oplus 
\opn{Hom}_A(\opn{Res}_A(R),\opn{Res}_A(R)) 
\]
we may write
\[
\Psi = \begin{bmatrix}
\Phi_1 & 0\\
0 & \Phi_2
\end{bmatrix}
\]
Hence, the fact that $\mrm{H}(\Phi_1)$ and $\mrm{H}(\Phi_2)$ are both bijective implies that $\mrm{H}(\Psi)$ is also bijective, showing that $\Psi$ is an isomorphism.
Hence, $\injdim_B(B) < \infty$.
\end{proof}

\begin{cor}\label{cor:existDGFFD}
Let $A$ be a noetherian ring which has a dualizing complex.
Then there exists a noetherian DG-ring $B$ with bounded cohomology such that the following hold:
\begin{enumerate}
\item There is an equality $\mrm{H}^0(B) = A$.
\item The natural map $\pi_B:B \to \mrm{H}^0(B)$ has a section.
\item It holds that $\opn{FPD}(B) < \infty$.
\end{enumerate}
\end{cor}
\begin{proof}
By assumption, there exists a dualizing complex $R$ over $A$,
and since by \cite[Theorem 3.9]{Yek} any shift of a dualizing complex is a dualizing complex, we may assume that $\sup(R) < 0$. 
Letting $B = A \skewtimes R$,
by \cref{thm:trivialExt} it holds that $\injdim_B(B) < \infty$,
so \cref{thm:injDimFFD} implies that $\opn{FPD}(B) < \infty$.
\end{proof}

In the next section, 
we will also require some results about flat modules and flat dimension in the presence of a dualizing complex,
which we now discuss.

\begin{prop}\label{prop:flatProj}
Let $A$ be a noetherian ring which has a dualizing complex.
Then there exist an integer $N$ such that for all $A$-modules $M$ with $\flatdim_A(M) < \infty$,
it holds that $\projdim_A(M) \le \flatdim_A(M) + N$.
\end{prop}
\begin{proof}
This is contained in \cite[Theorem]{Jor05}.
The number $N$ is the injective dimension of a dualizing complex $R$ with the property that $\inf(R) = 0$.
\end{proof}

We denote, for a ring $A$, by $\opn{FFD}(A)$ the finitistic flat dimension of $A$. By definition, 
\[
\opn{FFD}(A) = \sup\{\flatdim_A(M) \mid M\in \opn{Mod}(A), \flatdim_A(M) < \infty\}.
\]

The above result of J{\o}rgensen implies that: 
\begin{cor}\label{cor:FPDisFFD}
Let $A$ be a noetherian ring which has a dualizing complex.
Then $\opn{FPD}(A) < \infty$ if and only if $\opn{FFD}(A) < \infty$.
\end{cor}
\begin{proof}
If $\opn{FPD}(A) = L <\infty$ and $M$ is any $A$-module of finite flat dimension,
then by \cref{prop:flatProj} we know that $\projdim_A(M) < \infty$,
so $\flatdim_A(M) \le \projdim_A(M) \le L$, which shows that $\opn{FFD}(A) \le L$.
Conversely, if $\opn{FFD}(A) = L < \infty$,
and $M$ is an $A$-module of finite projective dimension,
then using \cref{prop:flatProj} again, 
we see that
\[
\projdim_A(M) \le \flatdim_A(M) + N \le L + N
\]
so $\opn{FPD}(A) \le L+N < \infty$.
\end{proof}

Following \cite{InKr},
for a ring $A$, we let
\[
d(A) = \sup\{\projdim_A(F) \mid F \mbox{ is a flat $A$-module}\}.
\]
It follows immediately from \cref{prop:flatProj} that
\begin{cor}\label{cor:daFinite}
Let $A$ be a noetherian ring which has a dualizing complex.
Then $d(A) < \infty$.
\end{cor}

Finally, we require the next result that relies heavily on our differential graded methods, and is used crucially in the next section.

\begin{thm}\label{thm:sequence}
Let $A$ be a noetherian ring which has a dualizing complex $S$.
Then the following are equivalent:
\begin{enumerate}
\item $\opn{FPD}(A) = +\infty$.
\item There exists an increasing sequence $a_n$ such that $a_n \in \mathbb{N}$ for all $n$, $\lim_{n\to +\infty} a_n = +\infty$,
and for each $n$, there is a left $A$-module $M_n$ with $\flatdim_A(M_n) = a_n$,
and $\opn{Tor}^A_{a_n}(R,M_n) \ne 0$ with $R$ a dualizing complex over $A$, of the form $R = S[j]$ for some $j \in \mathbb{Z}$.
\end{enumerate}
\end{thm}
\begin{proof}
By \cref{cor:FPDisFFD} the condition that $\opn{FPD}(A) = +\infty$ is equivalent to $\opn{FFD}(A) = +\infty$,
and clearly (2) implies it.
Assume, conversely, that $\opn{FPD}(A) = +\infty$,
so we also have that $\opn{FFD}(A) = +\infty$.
Let $R$ be a dualizing complex over $A$ with $\sup(R) < 0$.
Following \cref{cor:existDGFFD}, 
let $B = A \skewtimes R$,
and let $K = \opn{FPD}(B) < \infty$.
Because $\opn{FFD}(A) = +\infty$,
we may find an increasing sequence $b_n$ such that $b_n \in \mathbb{N}$ for all $n$, $\lim_{n\to +\infty} b_n = +\infty$,
and for each $n$, there is a left $A$-module $M_n$ with $\flatdim_A(M_n) = b_n$.
For each $n$, consider the left DG-module
$N_n = B\otimes^{\mrm{L}}_A M_n$.
By \cref{prop:flatProj} we know that $\projdim_A(M_n) < \infty$,
which implies that
\[
\projdim_B(N_n) = \projdim_{\mrm{H}^0(B)}(\mrm{H}^0(B)\otimes^{\mrm{L}}_B N_n) = \projdim_A(M_n) < \infty.
\]
Hence, by the definition of the finitistic projective dimension over $B$, it follows that
\[
\projdim_B(N_n) + \inf(N_n) \le K.
\]
Since $b_n = \flatdim_A(M_n) \le \projdim_A(M_n) = \projdim_B(N_n)$,
we deduce that
$b_n + \inf(N_n) \le K$.
On the other hand, the fact that $N_n$ is given by a tensor product over $A$, and the definition of flat dimension implies that 
\[
\inf(N_n) \ge \inf(B) - b_n = \inf(R) - b_n.
\]
Thus, for all $n$ it holds that
\[
\inf(R) - b_n \le \inf(N_n) \le K - b_n.
\]
Since $-\infty < \inf(R) < \infty$ and $K < \infty$ are fixed constants, there exist some $\inf(R) \le j \le K$ for which the value of $\inf(N_n)$ is equal to $j-b_n$ infinitely many times.
This implies the existence of a sequence $a_n$,
given as a subsequence of $b_n$,
for which, for all $n$, 
it holds that $\flatdim_A(M_n) = a_n$,
and $\inf(N_n) = j-a_n$.
Since over $A$ there is an isomorphism $B \cong A \oplus R$,
it follows that
\[
\inf(N_n) =  \inf(B\otimes^{\mrm{L}}_A M_n) = \inf((A \oplus R) \otimes^{\mrm{L}}_A M_n) = \inf(R \otimes^{\mrm{L}}_A M_n).
\]
It follows that
\[
\inf(R[j] \otimes^{\mrm{L}}_A M_n) = \inf(N_n[j]) = \inf(N_n) - j = -a_n.
\]
This implies that $\opn{Tor}^A_{a_n}(R[j],M_n) \ne 0$,
and since $R[j]$ is a dualizing complex over $A$,
we are done.
\end{proof}

\section{Rickard's theorem and dualizing complexes}

Given a triangulated category $\mathcal{T}$,
and a set $\mathbf{S}$ of objects of $\mathcal{T}$,
recall that the localizing subcategory generated by $\mathbf{S}$ is defined to be the smallest full triangulated subcategory $\mathcal{S} \subseteq \mathcal{T}$ such that $\mathbf{S}$ is contained in $\mathcal{S}$, and such that $\mathcal{S}$ is closed under infinite coproducts. 
If $\mathcal{S} = \mathcal{T}$,
we say that $\mathbf{S}$ generates $\mathcal{T}$.

As explained in the proof of \cite[Theorem 4.4]{Rickard},
if $\mathcal{T} = \cat{D}(A)$ for a ring $A$,
it follows from Bousfield localization (\cite{Bous,Kra10}) that if $\mathbf{S} = \{M\}$ is a single object,
then $M$ generates $\cat{D}(A)$ if and only if
any $X \in \cat{D}(A)$ for which
$\opn{Hom}_{\cat{D}(A)}(M,X[n]) = 0$ for all $n\in \mathbb{Z}$ satisfies that $X \cong 0$.

For a ring $A$, we denote by $\opn{Proj}(A)$ the set of all projective $A$-modules, 
by $\opn{Inj}(A)$ the set of all injective $A$-modules,
and by $\opn{Flat}(A)$ the set of all flat $A$-modules.
The homotopy categories of complexes of projective, injective and flat modules are denoted by $\cat{K}(\opn{Proj}(A))$, $\cat{K}(\opn{Inj}(A))$ and $\cat{K}(\opn{Flat}(A))$ respectively. These are triangulated categories. See \cite{Jor2,Kra05,Neeman2008,Stov} for more details about these categories.

In \cite[Theorem 4.3]{Rickard},
Rickard, answering a question of Keller, showed that if $A$ is a finite dimensional algebra over a field, and if $\opn{Inj}(A)$ generates $\cat{D}(A)$,
then $\opn{FPD}(A) < \infty$.
Our next results uses results of the previous sections to generalize this result to a much larger class of rings,
namely, the noetherian rings which have a dualizing complex.

\begin{thm}\label{thm:rickard}
Let $A$ be a ring which is left noetherian and right noetherian, and has a dualizing complex $R$.
Assume that $\opn{Res}_A(R)$ generates $\cat{D}(A)$.
Then $\opn{FPD}(A) < \infty$.
\end{thm}
\begin{proof}
Suppose that $\opn{FPD}(A) = \infty$.
According to \cref{thm:sequence},
there exist an increasing sequence $a_n$ such that $a_n \in \mathbb{N}$ for all $n$, $\lim_{n\to +\infty} a_n = +\infty$,
and for each $n$, there is a left $A$-module $M_n$ with $\flatdim_A(M_n) = a_n$,
and $\opn{Tor}^A_{a_n}(R[j],M_n) \ne 0$ for some fixed $j$.
The fact that $R$ generates $\cat{D}(A)$ implies that $R[j]$ also generates $\cat{D}(A)$, 
so let us replace $R$ by $R[j]$,
and suppose that for all $n$ it holds that
$\opn{Tor}^A_{a_n}(R,M_n) \ne 0$.
For every $n$, 
let $P_n$ be a flat resolution of $M_n$ of length $a_n$.
Then $\mrm{H}^i(P_n[-a_n]) = 0$ for all $i \ne a_n$,
and $\mrm{H}^{a_n}(P_n[-a_n])) = M_n$.
This implies, as in the proof of \cite[Theorem 4.3]{Rickard},
that the natural inclusion
\[
\Phi: \bigoplus_n P_n[-a_n] \to \prod_n P_n[-a_n]
\]
is a quasi-isomorphism.
However, as in loc. cit, we claim that $\Phi$ is not a homotopy equivalence. 
The proof in loc. cit relies on the fact that $A$ is a finite dimensional algebra over a field, so here our proof diverges from it.
Considering the complex $\opn{Rest}_{A^{\op}}(R)$,
since its cohomology is bounded and finitely generated over $A^{\op}$,
the fact that $A$ is right noetherian implies by \cite[Proposition 7.4.9]{YeBook} that we may replace $\opn{Rest}_{A^{\op}}(R)$ by a complex $S$,
isomorphic to it in $\cat{D}(A^{\op})$,
with the property that $S$ is a bounded complex of finitely generated right $A$-modules.
The noetherian property then implies that each $S^i$ is a finitely presented right $A$-module,
so the functor $S^i\otimes_A -$ commutes with infinite products for all $i$.
By definition of the tensor product of complexes,
for any complex $T$ of left $A$-modules, 
it holds that
\[
(S \otimes_A T)^n = \bigoplus_{p + q = n} S^p\otimes_A T^q
\]
for any $n \in \mathbb{Z}$.
Since $S^i \ne 0$ only for finitely many $i$'s,
the above direct sum is finite,
which implies that the functor $S\otimes_A -$ also commutes with infinite direct products.
We now apply the functor $S\otimes_A -$ to $\Phi$.
If $\Phi$ was a homotopy equivalence, 
then $S \otimes_A \Phi$ was also an homotopy equivalence,
so that $\mrm{H}^0(S \otimes_A \Phi)$ would have to be an isomorphism.
For any $n$, we have that
\[
\mrm{H}^0(S\otimes_A P_n[-a_n]) = \mrm{H}^{-a_n}(S\otimes_A P_n) = \opn{Tor}^A_{a_n}(R,M_n) \ne 0.
\]
The fact that the functor $\mrm{H}^0(S\otimes_A -)$ commutes with both products and coproducts implies that the map
$\mrm{H}^0(S \otimes_A \Phi)$ is isomorphic to the natural map
\[
\bigoplus_n \opn{Tor}^A_{a_n}(R,M_n) \to \prod_n \opn{Tor}^A_{a_n}(R,M_n),
\]
and since this map is clearly not an isomorphism,
we deduce that $\Phi$ is not a homotopy equivalence.
Since $A$ is left coherent,
it follows by \cite[Theorem 2.1]{Chase} that $\prod_n P_n[-a_n]$ is a complex of flat $A$-modules.
Clearly, this is also the case for $\bigoplus_n P_n[-a_n]$.
This allows us to consider $\Phi$ as a morphism in the triangulated category $\cat{K}(\opn{Flat} A)$.
Let us embed it in a distinguished triangle 
\[
\bigoplus_n P_n[-a_n] \xrightarrow{\Phi} \prod_n P_n[-a_n] \to D \to \left(\bigoplus_n P_n[-a_n]\right)[1]
\]
in $\cat{K}(\opn{Flat} A)$,
where $D$ is the mapping cone of $\Phi$.
The fact that $\Phi$ is a quasi-isomorphism implies that $\mrm{H}^n(D) = 0$ for all $n \in \mathbb{Z}$,
while the fact that $\Phi$ is not a homotopy equivalence implies that $D \ncong 0$ in $\cat{K}(\opn{Flat} A)$.
By definition of the mapping cone, 
it follows that $D^n = 0$ for all $n<-1$.
We now wish to replace $D$ by a complex of projective modules.
According to \cite[Proposition 8.1]{Neeman2008},
the inclusion functor $\opn{inc}:\cat{K}(\opn{Proj}(A)) \to \cat{K}(\opn{Flat}(A))$ has a right adjoint. We denote it, following \cite{InKr}, by $q:\cat{K}(\opn{Flat}(A)) \to \cat{K}(\opn{Proj}(A))$.
Let $C = q(D) \in \cat{K}(\opn{Proj}(A))$. We make three claims about this complex of projective $A$-modules.
First, we claim that $\mrm{H}^n(C) = 0$ for all $n \in \mathbb{Z}$.
This is because, this holds for $D$, and since by \cite[Theorem 2.7(1)]{InKr}
(which holds because of \cref{cor:daFinite}), there is a quasi-isomorphism $C \to D$.
Second, we note that by \cite[Theorem 2.7(2)]{InKr},
we have that $C^n = 0$ for all $n\le -2-d(A)$.
Finally, we claim that $C \ncong 0$ in $\cat{K}(\opn{Proj}(A))$.
To show this, we invoke \cite[Corollary 9.4]{Neeman2008} which says that  
$C = q(D)$ is null homotopic if and only if for any complex of right $A$-modules $L$ it holds that $L\otimes_A D$ is an acyclic complex.
But as we have seen above, 
the map $\mrm{H}^0(S\otimes_A \Phi)$ is not an isomorphism,
and since by definition $D$ is the cone of $\Phi$,
this implies that $S\otimes_A D$ is not acyclic.
Hence, $C = q(D) \ncong 0$ in $\cat{K}(\opn{Proj}(A))$.
We have thus shown that $C$ is an acyclic bounded below complex of projective $A$-modules which is not null homotopic.

We now use the fact that $R$ is a dualizing complex,
and invoke the covariant Grothendieck duality theorem.
According to \cite[Theorem 4.8]{InKr}
(see also \cite[Theorem 4.5]{Pos} for a more general result,
and \cite{Neeman2008} for a lengthy discussion of this result),
the functor
\[
R \otimes_A - : \cat{K}(\opn{Proj} A) \to \cat{K}(\opn{Inj} A)
\]
is an equivalence of triangulated categories.
Applying this equivalence of categories to the acyclicity of $C$,
we see that for all $n\in \mathbb{Z}$,
it holds that
\[
0 = \opn{Hom}_{\cat{K}(\opn{Proj} A)}(A,C[n]) = 
\opn{Hom}_{\cat{K}(\opn{Inj} A)}(R,R\otimes_A C[n])=
\opn{Hom}_{\cat{K}(A)}(R,R\otimes_A C[n]).
\]
However, the definition of the tensor product operation on complexes,
and the fact that $R$ is bounded and $C$ is bounded below implies that $R\otimes_A C[n]$ is a bounded below complex of injective $A$-modules. 
Hence, $R\otimes_A C[n]$ is K-injective, 
so there is an equality
\[
0 = \opn{Hom}_{\cat{K}(A)}(R,R\otimes_A C[n])
= \opn{Hom}_{\cat{D}(A)}(R,R\otimes_A C[n]).
\]
Finally, we observe that the fact that $C \ncong 0$ in $\cat{K}(\opn{Proj} A)$ implies that after applying the equivalence $R \otimes_A -$,
it is still the case that $R\otimes_A C \ncong 0$ in $\cat{K}(\opn{Inj} A)$,
and since $R\otimes_A C$ is K-injective,
this implies that $R\otimes_A C \ncong 0$ in $\cat{D}(A)$.
But we saw above that 
\[
\opn{Hom}_{\cat{D}(A)}(R,R\otimes_A C[n]) = 0
\]
for all $n \in \mathbb{Z}$.
We deduce that $R\otimes_A C$ does not belong to the localizing triangulated subcategory generated by $\opn{Res}_A(R)$,
which contradicts the assumption that $\opn{Res}_A(R)$ generates $\cat{D}(A)$.
Hence $\opn{FPD}(A) < \infty$.
\end{proof}

The above result assumed that $\cat{D}(A)$ is generated by a dualizing complex, a condition that might seem to be stronger than Rickard's condition that the injective modules generate $\cat{D}(A)$. 
The next result shows that these conditions are equivalent.
It is similar to \cite[Proposition 4.7]{InKr} which is in a commutative context.

\begin{prop}\label{prop:injGenerate}
Let $A$ be a left and right noetherian ring with a dualizing complex $R$.
Then the localizing subcategory of $\cat{D}(A)$ generated by $\opn{Res}_A(R)$ is equal to the localizing subcategory of $\cat{D}(A)$ generated by the injective $A$-modules:
\[
\opn{Loc}(\opn{Res}_A(R)) = \opn{Loc}(\opn{Inj}(A)).
\]
\end{prop}
\begin{proof}
Since $\opn{Res}_A(R)$ is a bounded complex of injective $A$-modules, 
it is clear that $\opn{Res}_A(R) \in \opn{Loc}(\opn{Inj}(A))$,
and hence, that $\opn{Loc}(\opn{Res}_A(R)) \subseteq \opn{Loc}(\opn{Inj}(A))$.
To show the converse, it is enough to show that for any injective $A$-module $I$, it holds that $I \in \opn{Loc}(\opn{Res}_A(R))$.
Let $F = \opn{Hom}_A(R,I)$.
It follows from \cite[Lemma 4.1(b)]{Pos} that $F$ is a bounded complex of flat $A$-modules.
Letting $P = q(F) \in \cat{K}(\opn{Proj}(A))$,
it follows from \cite[Theorem 2.7]{InKr} that $P$ is a bounded complex of projective $A$-modules.
By a repeated use of stupid truncation functors,
we may find a finite sequence of complexes of $A$-modules $M_1,M_2,\dots,M_n$
with the following properties:
\begin{enumerate}
\item  $M_1,M_2$ are shifts of free $A$-modules.
\item For each $2<i\le n$
there is a distinguished triangle of the form
\[
M_{i-2} \to M_{i} \to M_{i-1} \to M_{i-2}[1].
\]
\item There is an isomorphism $M_n \cong P$.
\end{enumerate}
Applying the functor $R\otimes_A -$ to this sequence,
we obtain a sequence of complexes of $A$-modules with the following properties:
\begin{enumerate}
\item  $R\otimes_A M_1, R\otimes_A M_2$ are shifts of direct sums of copies of $\opn{Res}_A(R)$.
\item For each $2<i\le n$
there is a distinguished triangle of the form
\[
R\otimes_A M_{i-2} \to R\otimes_A M_{i} \to R\otimes_A M_{i-1} \to R\otimes_A M_{i-2}[1].
\]
\item There is an isomorphism $R\otimes_A M_n \cong R\otimes_A P$.
\end{enumerate}
This shows that $R\otimes_A P \in \opn{Loc}(\opn{Res}_A(R))$,
and since by \cite[Theorem 4.8]{InKr} there is an isomorphism 
\[
R\otimes_A P = R\otimes_A q(\opn{Hom}_A(R,I)) \cong I,
\]
we deduce that $I \in \opn{Loc}(\opn{Res}_A(R))$.
\end{proof}

\begin{cor}\label{cor:main}
Let $A$ be a left and right noetherian ring which has a dualizing complex.
If the localizing subcategory of $\cat{D}(A)$ generated by the injective $A$-modules is equal to $\cat{D}(A)$ then $\opn{FPD}(A) < \infty$.
\end{cor}
\begin{proof}
This follows immediately from \cref{thm:rickard} and \cref{prop:injGenerate}. 
\end{proof}

The paper \cite{Rickard} contains several interesting examples of classes of rings for which the injectives generate $\cat{D}(A)$. 

In view of the above result, 
the reader may wonder if the condition that $A$ has a dualizing complex is necessary. 
The next two examples show that in the absence of a dualizing complex,
there is no relation between generation by injectives and having finite finitistic dimension.

\begin{exa}\label{exa:nessDual}
Let $A$ be a commutative noetherian ring of infinite Krull dimension.
For a concrete example of such an $A$, see \cite[Example A.1]{Nagata}.
Since $A$ is commutative and noetherian,
the Hopkins-Neeman classification \cite{Hopkins, Neeman1992} of localizing subcategories over commutative noetherian rings implies that $\opn{Loc}(\opn{Inj}(A)) = \cat{D}(A)$.
See \cite[Theorem 3.3]{Rickard} for details.
On the other hand, since $A$ has infinite Krull dimension,
it follows from \cite[Corollary 5.5]{Bass} that $\opn{FPD}(A) = \infty$.
It might be reassuring to mention that by \cite[Corollary V.7.2]{RD},
a commutative noetherian ring of infinite Krull dimension cannot have a dualizing complex.
\end{exa}

\begin{exa}
Let $A$ be a commutative noetherian local ring which does not have a dualizing complex. 
For a concrete example,
see for instance \cite[Example 6.1]{Ogoma}.
Since $A$ is local,
it has finite Krull dimension,
so by \cite{RG} it follows that $\opn{FPD}(A)<\infty$.
As in the previous example, injectives generate $\cat{D}(A)$.
\end{exa}

\begin{rem}
The paper \cite{Foxby} gave a direct proof of the fact that a commutative noetherian ring with a dualizing complex has finite finitistic dimension.
In view of \cite[Theorem 3.3]{Rickard},
this also follows from \cref{cor:main}.
\end{rem}

\textbf{Acknowledgments.}

The author thanks Isaac Bird, Jordan Williamson and Amnon Yekutieli for helpful discussions.
This work has been supported by the grant GA~\v{C}R 20-13778S from the Czech Science Foundation.

\bibliographystyle{abbrv}
\bibliography{main}

\end{document}